\documentclass[11pt, a4paper, reqno]{amsart}
\usepackage{amsmath, amsthm, amssymb, mathtools, microtype, xcolor}
\usepackage[utf8]{inputenc}

\theoremstyle{plain}
\newtheorem{theorem}{Theorem}
\newtheorem*{maintheorem}{Main Theorem}

\newtheorem{lemma}[theorem]{Lemma}
\newtheorem{corollary}[theorem]{Corollary}
\theoremstyle{definition}

\theoremstyle{remark}
\newtheorem{remark}[theorem]{Remark}

\def\R{\mathbb{R}}	
\renewcommand{\leq}{\leqslant} 		
\renewcommand{\geq}{\geqslant}

\usepackage[hidelinks]{hyperref}
\numberwithin{equation}{section}
\def\cA{\mathcal{A}}
\def\hcA{\hat\cA}
\def\cB{\mathcal{B}}
\def\hcB{\hat\cB}

\def\cF{\mathcal{F}}
\def\hcF{\hat\cF}

\def\from{\colon}
\def\into{\hookrightarrow}
\def\onto{\twoheadrightarrow}

\usepackage{parskip}
\usepackage{enumerate}

\begin{document}

\title{On the structure of homogeneous local Poisson brackets}

\author{Guido Carlet}
\address
{Université Bourgogne Europe, CNRS, IMB UMR 5584, 21000 Dijon, France}
\email{guido.carlet@ube.fr}

\author{Matteo Casati}
\address
{School of Mathematics and Statistics,
Ningbo University,
818 Fenghua Road, Jiangbei District,
Ningbo City, Zhejiang Province, People's Republic of China}
\email{matteo@nbu.edu.cn}

\subjclass[2010]{}

\date{\today}

\begin{abstract}
We consider an arbitrary Dubrovin-Novikov bracket of degree $k$, namely a homogeneous degree $k$ local Poisson bracket on the loop space of a smooth manifold $M$ of dimension $n$, and show that $k$ connections, defined by explicit linear combinations with constant coefficients of the standard connections associated with the Poisson bracket, are flat.
\end{abstract}

\maketitle

\tableofcontents

\section{Introduction}

In the early 1980s, Dubrovin and Novikov~\cite{dn83, dn84} introduced the problem of characterizing the structure and the deformations of homogeneous local Poisson brackets of degree $k$, which are since known as Dubrovin-Novikov brackets, differential-geometric Poisson brackets, or homogeneous Hamiltonian operators. In coordinates $u^1, \dots, u^n$ they can be written as
\begin{equation} \label{first}
\{u^i(x), u^j(y)\}^{[k]} = \sum_{s=0}^k P_l^{ij}(u, u_x, \dots ) \delta^{(s)}(x-y)
\end{equation}
with the coefficients $P_l^{ij}$ constrained by homogeneity in the number of $x$ derivatives, skew-symmetry, and Jacobi identity.

It is known since~\cite{n85} that a local homogeneous Poisson bracket should be studied as a finite-dimensional differential-geometric structure on a smooth manifold $M$. However, the geometric interpretation of the objects defined on $M$ by the coefficients $P_l^{ij}$ and of the constraints imposed on them by the skew-symmetry and the Jacobi identity is still not clear for an arbitrary degree $k$. The structure problem is that of a geometrically meaningful formulation of these objects on $M$, and of the differential constraints to which they are subject, aiming at the classification of homogeneous Poisson brackets on $M$. 

The importance of the homogeneous Poisson brackets is also due to the fact that they appear as leading term in the general form of dispersive Poisson brackets
\begin{equation}
\{ \cdot , \cdot \} = \sum_{l \geq k}  \{\cdot , \cdot\}^{[l]}
\end{equation}
where each homogeneous term $\{\cdot , \cdot\}^{[l]}$ is of the form~\eqref{first}. Higher degree terms can be regarded as a deformation of the leading-order homogeneous Poisson bracket of degree $k$. The basic structure describing such deformations is the Poisson cohomology of the leading homogeneous Poisson bracket. 

Despite several results in low degree $k$, for the most part limited to the case of $M=\R^n$ and often to small values of $n$, the structure and the deformation theory of homogeneous Poisson brackets of arbitrary degree $k$ remains mostly unknown. The only general result, proved by Doyle~\cite{d93}, is that the connection $\nabla^{(0)}$, associated with the coefficients of $u^{l,k} \delta(x-y)$ in~\eqref{first}, is symmetric and flat. 

The aim of this paper is to provide new insights on the structure of homogeneous Poisson brackets of arbitrary degree, laying the foundation to a sequence of works that we plan to devote to the study of their structure and deformations.
A common approach in the study of the structure of such brackets for low $k$ is to look for a special coordinate system in which the form of the brackets is particularly simple;    for example, for general $k$, by writing them in the flat coordinates for the connection $\nabla^{(0)}$. However, the use of special coordinates can obfuscate, rather than clarify, the geometric meaning of the equations. The philosophy of this work is therefore to study the structure of homogeneous Poisson brackets in an arbitrary coordinate system and for arbitrary $k$. 

This work is organized as follows: in the next section we state our main theorem; in Section~3 we formulate an auxiliary cohomology problem which leads to the proof of our main result; in Section~4 we review the low-degree brackets in the light of our main theorem; in the last section we announce some results and list some open problems.

\begin{remark}
The structure of local homogeneous Poisson brackets has been thoroughly studied in low degree $k\leq3$. Let us review the available literature. The ultralocal case $k=0$ does not concern us here, since it does not involve any connection; it corresponds to the case of a Poisson structure on $M$, see~\cite{m98}. The structure of the hydrodynamic type Poisson brackets is described in~\cite{dn83}. The case of degree $k=2$ has been studied in~\cite{n85, p86, d93, m98, f08, f09}. The degree $k=3$ case has been investigated in~\cite{n85, p86, p91, d93, p97, m98, bp01, fpv14, fpv16}. The study of Poisson brackets of degree $k=1$ with degenerate $g^{ij}$ has been undertaken in~\cite{g85, m98}, and treated for $n=2,3$ in~\cite{s16}. 
\end{remark}

\begin{remark}
The theory of local Poisson brackets considered in this paper allows for several extensions. If the coordinates $u^1,\ldots,u^n$ are regarded as functions of several independent variables $x_1, \dots , x_D$ we say that the Poisson bracket is multidimensional. This generalization was introduced in~\cite{dn84} and studied, for example, in~\cite{m88,c14}. Moreover, it is possible to study brackets defined by pseudodifferential operators, extending the theory to the \emph{non-local} case, that has been considered for example in~\cite{mf90, f91, cfpv19}. Homological algebra methods used in this paper have originally been applied to the computation of bihamiltonian cohomology groups~\cite{cps16, cps17} and Poisson cohomology groups in the multidimensional case~\cite{ccs17}.
\end{remark}

\subsection*{Acknowledgments}
This work has been supported by: the EIPHI Graduate School (contract ``ANR-17-EURE-0002'', grant ``METTHHOD''), the region ``Bourgogne Franche-Comté'' (ANER grant ``FROBENIUS''), the LMS Research Grant (Scheme 4 – Research in Pair) ``Poisson cohomology of differential geometric Hamiltonian operators'', the National Science Foundation of China (Grants no.~12101341 and no.~12431008), Ningbo City Yongjiang Innovative Talent Program, and Ningbo University Talent Introduction and Research Initiation Fund.

\section{Statement of the main theorem}

Let $M$ be a smooth $n$-dimensional manifold whose formal loop space is endowed with a homogeneous local Poisson bracket of degree $k$. More precisely, on a chart $U$ with coordinates $u^1, \dots , u^n$ on $M$, the Poisson bracket is formally given by the expression
\begin{equation} \label{kbivect}
\{ u^i(x) ,u^j(y) \}  = \sum_{s=0}^k P_s^{ij} \delta^{(s)}(x-y) 
\end{equation}
where $P_s^{ij} \in \cA_{k-s}$. We denote by $\cA$ the algebra of differential polynomials, which on $U$ are given by formal power series in the variables $u^{i, s}$ with $1\leq i \leq n$, $s \geq 1$ with coefficients which are smooth functions of $u^1, \dots , u^n$. The standard degree $\deg$ is defined on $\cA$ by assigning the degree $s$ to the generators $u^{i, s}$. We denote by $\cA_d$ the homogeneous component of standard degree $d$. 
The derivation $\partial_x$ on $\cA$ is given by 
\begin{equation}
\partial_x = \sum_{s\geq0} u^{i, s+1} \frac{\partial }{\partial u^{i, s}} ,
\end{equation} 
where we pose $u^{i,0}=u^i$.
Under coordinate transformations $\tilde{u}^i=\tilde{u}^i(u)$, the formal transformation rule
\begin{equation}
\{ \tilde{u}^i(x) ,\tilde{u}^j(y) \}  = \frac{\partial \tilde{u}^i}{\partial u^{i'}}(u(x)) \{ u^{i'}(x) ,u^{j'}(y) \}  \frac{\partial \tilde{u}^j}{\partial u^{j'}}(u(y))
\end{equation}
is equivalent, using the usual identities for the derivatives of the Dirac delta function, to 
\begin{equation} \label{gen-tr-rule}
\tilde{P}^{ij}_s = \sum_{t\geq 0} \binom{s+t}{s} 
\frac{\partial \tilde{u}^i}{\partial u^{i'}} 
P^{i'j'}_{s+t} \partial_x^t \left( \frac{\partial \tilde{u}^j}{\partial u^{j'}} \right).
\end{equation}
This equation should be understood as an identity between differential polynomials that induces the transformation rules for their coefficients. We name some coefficients by letting
\begin{equation} \label{coeff-def}
\{ u^i(x) ,u^j(y) \} = g^{ij}(u(x)) \delta^{(k)}(x-y) + \sum_{s=0}^{k-1} h_{(s)l}^{ij}(u(x)) u^{l, k-s}(x) \delta^{(s)}(x-y) + \dots  
\end{equation}
where the dots denote higher order terms of degree two or more in the variables $u^{i, s}$ with $s\geq1$, $i=1, \dots ,n$. 

The coefficients $g^{ij}$ define	 a $(2,0)$-tensor $g$ on $M$, with symmetry
\begin{equation*}
g^{ji} = (-1)^{k+1} g^{ij}.
\end{equation*}
The skew-symmetry of the bracket ~\eqref{coeff-def} also implies the following constraints on the coefficients $h_{(s)l}^{ij}$
\begin{equation} \label{skewh}
h^{ij}_{(s)l} = \sum_{t=0}^{k-1} (-1)^{t+1} \binom{t}{s} h^{ji}_{(t)l}+\binom{k}{s} \frac{\partial g^{ij}}{\partial u^l} 
\end{equation}
for $0\leq s\leq k-1$.

We assume that $g$ is nondegenerate, thereby imposing that $n$ is even for $k$ even. The coefficients $h_{(s)l}^{ij}$ are known to transform (up to a multiplicative constant) as contravariant Christoffel symbols; more precisely
\begin{equation} \label{gammadef}
\Gamma_{(s)ij}^{l} = - \binom{k}{s}^{-1} g_{i i'} h_{(s)j}^{i'l}
\end{equation}
transform as Christoffel symbols. They define $k$ connections $\nabla^{(s)}$ on the tangent bundle to $M$, i.e. maps 
\begin{equation}
\nabla^{(s)} : \Gamma(TM) \to \Omega^1(TM) 
\end{equation}
for $s=0, \dots , k-1$, where $\Omega^1(TM)= \Gamma(T^*M \otimes TM)$.  We call these the standard connections.  

We define linear combinations with constant coefficients of the standard connections $\nabla^{(s)}$ as follows
\begin{equation}  \label{quadconn}
\nabla^{[s]} = \sum_{t=0}^{s} c_s^t \nabla^{(t)}
\end{equation}
where $s=0, \dots , k-1$ and the constants are given by
\begin{equation}
c_s^t = (-1)^t \binom{k+s-t}{k} \binom{k}{t}.
\end{equation}
Notice that these constants are nonzero only if $t \leq s$ and $t \leq k$. Since for each fixed $s$ the constants $c_s^t$ sum to one, the formula above defines connections $\nabla^{[0]}, \dots, \nabla^{[k-1]}$ on $TM$. 

While it is known~\cite{d93} that $\nabla^{(0)}$ is flat, the remaining standard connections $\nabla^{(1)}, \dots , \nabla^{(k-1)}$ in general are not flat (see Remarks~\ref{nonflat1} and~\ref{nonflat2}). Surprisingly, the flatness holds for all the connections $\nabla^{[s]}$, as stated in our main result.
\begin{maintheorem}
The connections $\nabla^{[0]}, \dots, \nabla^{[k-1]}$ are flat. 
\end{maintheorem}

The first few flat connections are:
\begin{align}
\nabla^{[0]}  &= \nabla^{(0)}, \\
\nabla^{[1]}  &= - k \nabla^{(1)} +(k+1)   \nabla^{(0)}, \\
\nabla^{[2]}  &= \frac{k(k-1)}2 \nabla^{(2)} -k(k+1)   \nabla^{(1)} + \frac{(k+2)(k+1)}2  \nabla^{(0)},\\
&\ \ \vdots \\
\nabla^{[k-1]}  &= -(-1)^k k \nabla^{(k-1)} + \dots + \binom{2k-1}{k}  \nabla^{(0)}.
\end{align}

\begin{remark}
We say that a degree $k$ Poisson bracket is generic if the affine space spanned by the standard connections (or equivalently by the flat connections) is $k-1$ dimensional, namely if they are in general position. The study of different types of degenerations could provide interesting families of Poisson brackets. 
\end{remark}

\section{Proof of the main theorem}

The main difficulty in dealing with the constraints imposed by the Jacobi identity for arbitrary degree $k$ is finding an appropriate way to disentangle the large family of associated differential equations, which, if dealt with directly, become unmanageable already for quite low $k$. Our strategy is to encode the Jacobi identity as $d \circ d =0$, where $d$ is the adjoint action of the Poisson bivector $P$ associated with the homogeneous Poisson bracket on the differential complex of local multivector fields. For our purposes, it is actually sufficient to consider the associated differential complex $(\hcA, D_P)$ introduced by Liu and Zhang~\cite{lz11, lz13}. These computations also form the basis of our future work on the cohomology of homogeneous Poisson brackets.

Let us consider the algebra $\hcA$ of formal power series in the even variables $u^{i, s}$ with $1\leq i \leq n$, $s \geq 1$ and in the odd variables $\theta^s_i$ with $1\leq i \leq n$, $s \geq 0$ with coefficients which are smooth functions of the coordinates  $u^1, \dots , u^n$ in a chart $U$ on $M$. 
Clearly $\cA$ is a subalgebra of $\hcA$ and the standard degree $\deg$ extends to $\hcA$ by assigning degree $s$ to the generators $\theta_i^s$. We denote by $\hcA_d$ the homogeneous component of standard degree $d$. The derivation $\partial_x$ on $\hcA$ is given by 
\begin{equation}
\partial_x = \sum_{s\geq0} \left[ u^{i, s+1} \frac{\partial }{\partial u^{i, s}} + \theta_i^{s+1} \frac{\partial }{\partial \theta_i^s}  \right] 
\end{equation} 
where we denote $u^{i,0} = u^i$. 
The theta degree $\deg_\theta$ on $\hcA$ assigns degree one to the variables $\theta_i^s$ and zero to the remaining generators. We denote by $\hcA^p$ the homogeneous component of theta degree $p$ and $\hcA^p_d = \hcA^p \cap \hcA_d$. 
Define $\hcF = \hcA / \partial_x \hcA$ and denote $\int : \hcA \to \hcF$ the corresponding projection. The degrees defined above on $\hcA$ induce corresponding degrees on $\hcF$, and we denote the homogeneous components of $\hcF$ with the obvious upper and lower indices.

With the homogeneous local Poisson bracket of degree $k\geq1$ defined by~\eqref{kbivect} we associate the following element $P$ in $\hcF^2_k$
\begin{equation}
P =\int \tilde{P}, \qquad \tilde{P} =  \frac12 \sum_{s=0}^k P_s^{ij} \theta_i  \theta_j^s
\end{equation}
which in turn defines a superderivation of $\hcA$  
\begin{equation} \label{dp}
D_P = \sum_{s\geq0} \left[ \partial_x^s\left( \frac{\delta P}{\delta\theta_i} \right) \frac{\partial }{\partial u^{i,s}}  +  \partial_x^s \left( \frac{\delta P}{\delta u^i }\right) \frac{\partial }{\partial \theta_i^s}  \right].
\end{equation}
with $\deg D_P = k$ and $\deg_\theta D_P =1$. The variational derivatives are defined as 
\begin{equation}
\frac{\delta P}{\delta u^i} = \sum_{s\geq 0} (-\partial_x)^s \frac{\partial \tilde P}{\partial u^{i,s}}, \qquad \frac{\delta P}{\delta \theta_i} = \sum_{s\geq 0} (-\partial_x)^s \frac{\partial \tilde P}{\partial \theta^{s}_i} .
\end{equation}
The operator $D_P$ defines a differential on the complex $\hcA$.
\begin{lemma}[\cite{lz11, lz13}] 
The formula~\eqref{dp} defines a superderivation $D_P$ of $\hcA$ which squares to zero, i.e.
\begin{equation}
D_P \circ D_P = 0.
\end{equation}
\end{lemma}
The previous equation encodes the differential equations appearing in the Jacobi identity for~\eqref{kbivect}. To disentangle them we proceed to compute the cohomology of the complex $(\hcA, D_P)$, using spectral sequences. For a general introduction to this technique we refer for example to~\cite{m01}. A quick review can be found in Section 3 of~\cite{cps16}.
 
While the cohomology computations in general depend on the manifold $M$, in this context we are only interested in simplifying the operator $D_P$, therefore it is sufficient to consider the algebra $\hcA$ on a chart $U$ of $M$, that for simplicity we consider homeomorphic to a ball in $\R^n$. 

Let $\deg_u$ be the degree on $\hcA$ defined by imposing that $\deg_u u^{i,s} =1$ if $s >0$, and that the degree of all other generators is zero. We denote by $[.]_p$ the projection to the homogeneous component of $\deg_u$ equal to $p$. 
Denoting some coefficients in $P$ as in~\eqref{coeff-def}, we find the following preliminary formulas for some of their homogeneous components in degree $\deg_u$. 
\begin{lemma}
The following equations hold:
\begin{align}
[\tilde P]_0 &= \frac12 g^{ij} \theta_i \theta_j^k, \\
[\tilde P]_1 &= \frac12 \sum_{s=0}^{k-1} h_{(s) l}^{ij} u^{l, k-s} \theta_i \theta_j^{s}, \\
\left[  \frac{\delta P}{\delta \theta_i} \right]_0 
&=g^{ij} \theta_j^k,  \label{dpdt} \\
\left[  \frac{\delta P}{\delta u^i} \right]_0 
&= \frac12 \frac{\partial g^{lj}}{\partial u^i}  \theta_l \theta_j^k 
+ \frac12 \sum_{s=0}^{k-1} (-1)^{k-s} h_{(s)i}^{lj} \partial_x^{k-s} \left( \theta_l \theta_j^s \right).  \label{dpdu}
\end{align}
\end{lemma}
\begin{proof}
The first two are straightforward computations. For the third one observe that
\begin{equation}
\left[  \frac{\delta P}{\delta \theta_i} \right]_0 = \left[  \frac{\delta [\tilde P]_0}{\delta \theta_i} \right]_0,
\end{equation}
while for the fourth one 
\begin{equation}
\left[  \frac{\delta P}{\delta u^i} \right]_0 =\left[  \frac{\delta [\tilde P]_0}{\delta u^i} \right]_0 +\left[  \frac{\delta [\tilde P]_1}{\delta u^i} \right]_0,
\end{equation}
from which the above expressions easily follow.
\end{proof}

Let us decompose the differential $D_P = D_{-1} + D_0 + \dots$ in homogeneous components of degree $\deg_u D_s= s$. 
\begin{lemma}
The lowest degree homogeneous component of $D_P$  is given by
 \begin{equation}
D_{-1} = \sum_{s\geq 1} g^{ij} \theta_j^{k+s} \frac{\partial }{\partial u^{i,s}} . 
\end{equation}
\end{lemma}
\begin{proof}
Observe that
\begin{equation}
D_{-1}  = [D_P]_{-1} = \sum_{s\geq1}  \left[ \partial_x^s\left(\frac{\delta P}{\delta \theta_i} \right) \right]_0 \frac{\partial }{\partial u^{i,s}} 
= \sum_{s\geq1}  \left[ \partial_x^s\left[\frac{\delta P}{\delta \theta_i} \right]_0 \right]_0 \frac{\partial }{\partial u^{i,s}}
\end{equation}
which, after substituting~\eqref{dpdt}, gives the desired result.
\end{proof}

We consider on the complex $(\hcA, D_P)$ the compatible descending filtration $F \hcA$ induced by the degree $\deg_u + \deg_\theta$. More explicitly, $F^p \hcA^q$ includes the monomials with $\deg_\theta$ equal to $q$ and $\deg_u$ bigger or equal to $p-q$. Therefore $D_P$ maps $F^p \hcA^q$ to $F^p \hcA^{q+1}$. Let us denote by $(E_r^{pq}, d_r)_{r\geq0}$ the associated spectral sequence. 

On the page zero of the spectral sequence 
\begin{equation}
E_0^{pq} = \frac{F^p \hcA^{p+q}}{F^{p+1} \hcA^{p+q}} \simeq [\hcA^{p+q}]_{-q}
\end{equation}
the differential $D_P$ induces the differential $d_0: E_0^{p,q} \to E_0^{p,q+1}$, which can be identified with $D_{-1}$ acting on $\hcA$. Let us compute its cohomology. 

Let $\hcB$ be the ring of polynomials in the odd variables $\theta^s_i$ for $0 \leq s \leq k$ and  $1 \leq  i \leq n$ with coefficients given by smooth functions in the variables $u^1, \dots , u^n$:
\begin{equation}
\hcB = C^\infty(U) \left[ \{ {\theta}_i^s, \ 0 \leq s \leq k, \ 1 \leq i \leq n \}  \right].
\end{equation}

Let $i_{\hcB} : \hcB \into \hcA$ be the inclusion map and 
$\pi_{\hcB} \from \hcA \onto \hcB$ 
the projection which sends to zero the generators $u^{i,s}$ and $\theta_i^{s+k}$ for $s\geq1$, $1 \leq  i \leq n$. Clearly $\pi_{\hcB} \circ i_{\hcB} = 1_{\hcB}$, while $i_{\hcB} \circ \pi_{\hcB}$ is homotopic to the identity map $1_{\hcA}$. The cochain homotopy map $h_{\hcA}$ is the usual  one in de Rham theory, which in this case can be written as 
\begin{equation}
h_{\hcA} = \frac1l \sum_{s\geq 1} u^{i,s} g_{ji} \frac{\partial }{\partial \theta^{k+s}_j }
\end{equation}
on monomials with the degree $l\geq1$ in the generators $u^{i,s}$ and $\theta_i^{s+k}$ for $s\geq1$, $1 \leq  i \leq n$, while $h_{\hcA}$ is zero on monomials with degree $l=0$. We have that
\begin{equation}
D_{-1} \circ h_{\hcA} + h_{\hcA} \circ D_{-1} = 1_{\hcA} - i_{\hcB} \circ \pi_{\hcB} .
\end{equation} 

We obtain the following formal Poincaré lemma.
\begin{lemma}
The cohomology $H(\hcA, D_{-1})$ is isomorphic to $\hcB$ .
\end{lemma}

Notice that the lemma follows essentially from the contractibility of the fibers of the jet space, with coordinates $u^{i, s}$, $s\geq 1$, once we identify $D_{-1}$ with a de~Rham type operator. 

The differential $D_P$ induces a differential $d_1 \from E_1^{p,q} \to E_1^{p+1,q}$ on the first page of the spectral sequence $E_1 \simeq \hcB$. On $\hcB$ it is represented by
\begin{equation}
d_1 = \pi_{\hcB} \circ D_0 \circ i_{\hcB}.
\end{equation}
Removing the terms in $D_P$ that act on $u^{i,s}$ or on $\theta_i^{k+s}$ with $s>0$ we obtain
\begin{align}
d_1 &= \pi_{\hcB} \circ \left[  
\frac{\delta P}{\delta \theta_i} \frac{\partial }{\partial u^i} 
+ \sum_{s=0}^k \partial^s \left( \frac{\delta P}{\delta u^i} \right) \frac{\partial }{\partial \theta_i^s} \right]_0  
  \circ i_{\hcB} \\
&= \pi_{\hcB} \circ \left( 
\left[
\frac{\delta P}{\delta \theta_i}  \right]_0 \frac{\partial }{\partial u^i} 
+ \sum_{s=0}^k \left[ \partial^s \left( \left[ \frac{\delta P}{\delta u^i} \right]_0 \right) \right]_0 \frac{\partial }{\partial \theta_i^s}  \right)
  \circ i_{\hcB} .
\end{align}
By substitution of~\eqref{dpdt} and~\eqref{dpdu}, we have the following:
\begin{lemma}
The first page $E_1$ of the spectral sequence is isomorphic to the space $\hcB$ with differential 
\begin{align} \label{d1}
d_1 &= g^{ij} \theta^k_j \frac{\partial }{\partial u^i}   
-\frac12 \frac{\partial g^{ij}}{\partial u^l} \theta_j^k  \sum_{s=0}^k \theta_i^s \frac{\partial }{\partial \theta_l^s} \\ \notag
&+\frac12 \sum_{0\leq s \leq r \leq k} \left[ \sum_{t=0}^{k-1} (-1)^{k-t} \binom{k+s-t}{r} h_{(t)l}^{ij}  \right]\theta^r_i \theta_j^{k+s-r} \frac{\partial }{\partial \theta_l^s} .
\end{align}
\end{lemma}

\begin{remark}
We introduce the notation
\begin{equation} \label{handg}
h^{ij}_{(k) l} = \frac{\partial g^{ij}}{\partial u^l},
\end{equation} 
so that the formula~\eqref{d1} can be rewritten in the more compact form
\begin{equation}
d_1 = g^{ij} \theta^k_j \frac{\partial }{\partial u^i}  
+\frac12 \sum_{\substack{0\leq s \leq r \leq k \\ 0\leq t \leq k} }
 (-1)^{k-t} \binom{k+s-t}{r} h_{(t)l}^{ij} 
\theta^r_i \theta_j^{k+s-r} \frac{\partial }{\partial \theta_l^s} .
\end{equation} 
\end{remark}

To obtain the second page $E_2$ of the spectral sequence we need to compute the cohomology of the complex $(\hcB, d_1)$. 

Let us now consider the degree $\deg_{\theta^k}$ on $\hcB$ defined by assigning degree one to the variables $\theta^k_1, \dots \theta_n^k$ and zero to the remaining ones. It turns out that the operator $d_1$ is concentrated in degrees $0$ and $1$.  
We collect in the following lemma some observations about the form of $d_1$.
\begin{lemma} \label{lemma10}
The operator $d_1$ on $\hcB$  has only two homogeneous components
\begin{equation} 
d_1 = d_1^{(1)} + d_1^{(0)},
\end{equation}
with $\deg_{\theta^k} d_1^{(s)} = s$.
The operator $d_1^{(1)}$ can be written
\begin{multline} \label{d11-1}
d_1^{(1)}  = g^{ij} \theta^k_j \frac{\partial }{\partial u^i}   
+\frac12 \sum_{t=0}^{k}  (-1)^{k-t}  \binom{2k-t}{k} h_{(t)l}^{ij} \theta_i^k \theta_j^k \frac{\partial }{\partial \theta_l^k} \\
+\sum_{0\leq s ,t \leq k-1}  
(-1)^k \binom{k}{t}^{-1}
c_s^t h_{(t)l}^{ij}
\theta^k_i \theta_j^{s} \frac{\partial }{\partial \theta_l^s} 
\end{multline}
where
\begin{equation} \label{cts}
c_s^t = (-1)^t \binom{k+s-t}{k} \binom{k}{t}.
\end{equation}
The constants $c_s^t$, for $0 \leq s, t \leq k-1$,  form a lower triangular matrix $c$, i.e. $c_s^t=0$ if $s < t$, and satisfy $\sum_{t=0}^{k-1} c_s^t = 1$. The inverse $c^{-1}$ is a lower triangular matrix given by 
\begin{equation}
(c^{-1})_s^t = (-1)^{t} \binom{k}{s}^{-1} \binom{k+1}{s-t}
\end{equation} 
such that  $\sum_{t=0}^{k-1} (c^{-1})_s^t = 1$.
The operator $d_1^{(0)}$ is given by
\begin{align} \label{d10}
d_1^{(0)}  &= \frac12 \sum_{0\leq s <r < k} \left[ \sum_{t=0}^{k-1} (-1)^{k-t} \binom{k+s-t}{r} h_{(t)l}^{ij}  \right]\theta^r_i \theta_j^{k+s-r} \frac{\partial }{\partial \theta_l^s} .
\end{align}
\end{lemma}

\begin{proof}
A straightforward computation gives the homogeneous component of degree zero~\eqref{d10}, while for the homogeneous component of degree one we get:
\begin{align} \label{d11}
d_1^{(1)}  &= g^{ij} \theta^k_j \frac{\partial }{\partial u^i}   
-\frac12 \frac{\partial g^{ij}}{\partial u^l} \theta_j^k  \sum_{s=0}^k \theta_i^s \frac{\partial }{\partial \theta_l^s} \\ \notag
&+\frac12 \sum_{t=0}^{k-1}  (-1)^{k-t}  \binom{2k-t}{k} h_{(t)l}^{ij} \theta_i^k \theta_j^k \frac{\partial }{\partial \theta_l^k} \\ 
&+\frac12 \sum_{0\leq s ,t \leq k-1}  (-1)^{k-t} \left[ \binom{k+s-t}{k} h_{(t)l}^{ij} - \binom{k+s-t}{s} h_{(t)l}^{ji} \right] \theta^k_i \theta_j^{s} \frac{\partial }{\partial \theta_l^s} \notag.
\end{align}

We can substitute~\eqref{skewh} in the third line of ~\eqref{d11}, recalling that $h^{ij}_{(k)l}$ is defined as in~\eqref{handg}, to obtain
\begin{multline}
+\frac12 \sum_{0\leq s ,t \leq k-1}  (-1)^{k-t} \left[ \binom{k+s-t}{k} h_{(t)l}^{ij} \right. \\
\left. - \binom{k+s-t}{s}
 \sum_{r=0}^k (-1)^{r+1} \binom{r}{t} h^{ij}_{(r)l} 
\right] \theta^k_i \theta_j^{s} \frac{\partial }{\partial \theta_l^s} .
\end{multline}
In this expression we can isolate the terms containing $h_{(k)l}^{ij}$, which are equal to
\begin{equation} \label{combi}
-\frac12 \sum_{0\leq s \leq k-1}  \left[ 
\sum_{0\leq t \leq k-1}  
(-1)^{t+1}  \binom{k+s-t}{s}  \binom{k}{t} 
\right] h_{(k)l}^{ij}
 \theta^k_i \theta_j^{s} \frac{\partial }{\partial \theta_l^s} .
 \end{equation}
Thanks to the standard binomial identity (see Equation (5.25) in \cite{knuth})
\begin{equation} \label{knuth-id}
\sum_{k\leq l}(-1)^k\binom{l-k}{m}\binom{s}{k-n} =(-1)^{m+l}\binom{s-m-1}{l-m-n}, 
\end{equation}
we can easily see that 
\begin{equation}
\sum_{0 \leq t \leq k} (-1)^{t+1} \binom{k+s-t}{s} \binom{k}{t} = 0.
\end{equation}
It follows that expression~\eqref{combi} combines with the last term in the first line of~\eqref{d11}, giving a term which can be absorbed by the second line in~\eqref{d11} by extending the sum to $t=k$. 
Notice that such extra term is of the form 
$h^{ij}_{(k)l} \theta^k_i \theta^k_j$, hence it is vanishing for $k$ odd for symmetry considerations. 
We have thus obtained equation~\eqref{d11-1}
with
\begin{equation} \label{cts1}
c_s^t = \frac12 (-1)^t \binom{k}{t} \left[ \binom{k+s-t}{k} 
- \sum_{r=0}^{k-1} (-1)^{r+1} \binom{k+s-r}{s} \binom{t}{r} \right].
\end{equation}
Finally let us prove that these constants can be written as in~\eqref{cts}. Since $t<k$ the sum in~\eqref{cts1} can be extended to $0\leq r \leq k+s$. Then~\eqref{cts} follows from the identity~\eqref{knuth-id} with $k=r$, $l=k+s$, $m=s$, $s=t$ and $n=0$: first we obtain that the sum is equal to $(-1)^k\binom{t-s-1}{k}$, which is equal to the first entry of the bracket in~\eqref{cts1} by the following well-known identity
\begin{equation}
\binom{n}{k}=(-1)^k\binom{k-n-1}{k}.
\end{equation}
The remaining properties of the coefficients $c_s^t$ can be easily proved from the standard identities of the binomial coefficients. 
\end{proof}

Since $d_1$ squares to zero and it decomposes in two homogeneous components in $\deg_{\theta^k}$, we have 
\begin{equation}
\left(d_1^{(1)}\right)^2 = 0, \qquad 
d_1^{(1)} d_1^{(0)} + d_1^{(0)} d_1^{(1)} =0, \qquad  
\left(d_1^{(0)}\right)^2 = 0.
\end{equation}

Setting $\theta_i^k = g_{ij} du^j$, the ring $\hcB$ can be identified with the ring of polynomials in $\theta_i^s$ with $0\leq s<k$ and $i=1, \dots, n$ with coefficients which are exterior forms in the variables $u^i$, namely
\begin{equation}
\hcB \cong \Omega(U) \left[ \{ {\theta}_i^s, \ 0 \leq s \leq k-1, \ 1 \leq i \leq n \}  \right].
\end{equation}
Let us define
\begin{equation} \label{gamma-s-def}
\Gamma_{[s]ij}^l := \sum_{t=0}^{k-1} c_s^t \Gamma_{(t)ij}^l.
\end{equation}
Under the above identification, we have that:
\begin{lemma}
The operator $d_1^{(1)}$ acts on $\hcB$ as 
\begin{equation} \label{d11-2}
d_1^{(1)} = d + \sum_{s=0}^{k-1} \Gamma_{[s]il}^j  du^i \theta_j^s \frac{\partial }{\partial \theta_l^s} 
\end{equation}
where $d$ is the standard exterior derivative in the coordinates $u^i$, $i=1, \dots , n$.
\end{lemma}

\begin{proof}
We have that $d_1^{(1)} u^i = g^{il} \theta_l^k = du^i$ and, since $d_1^{(1)}$ squares to zero, $d_1^{(1)} du^i =0$. Therefore the first two terms in~\eqref{d11-1} are simply given by the exterior derivative $d$, while the last term can be rewritten as in~\eqref{d11-2} by substitution of~\eqref{gamma-s-def} and~\eqref{gammadef}.
\end{proof}

For a fixed $s=0, \dots , k-1$, let $E_s$ be the trivial vector bundle over $U$ with fibre the vector space spanned by $\theta_1^s, \dots , \theta_n^s$. The space of $E_s$-valued differential forms $\Omega(U, E_s)$ is identified with the subspace $\hcB_s$ of $\hcB$ given by 
\begin{equation}
\hcB_s := \Omega(U) \langle \theta_1^s, \cdots, \theta_n^s \rangle.
\end{equation}

The operator $d_1^{(1)}$ defines a connection $\nabla^{[s]}$ on the bundle $E_s$, namely a map $\nabla^{[s]}: \Omega^p(U, E_s) \to \Omega^{p+1}(U, E_s)$. Clearly $\nabla^{[s]}$ squares to zero, i.e. it is a flat connection.

This concludes the proof of the main theorem. \hfill\qedsymbol

\section{Poisson brackets of low degree} \label{loworder}

We investigate the consequences of our main theorem in low degree, comparing them with some well-known results in the literature. 

The brackets of degree $k=0$ are also called ultralocal Poisson brackets~\cite{m98} and correspond to Poisson structures on the manifold $M$. They are not relevant to our discussion, so we will not review them here. 

\subsection{Poisson brackets of hydrodynamic type}

The degree $k=1$ Poisson brackets have the local form
\begin{equation} \label{pb-deg-1}
\{ u^i(x) ,  u^j(y) \} = g^{ij} \delta'(x-y) + b^{ij}_l u^l_x \delta(x-y)
\end{equation}
where we assume that $g$ is non-degenerate. The connection $\nabla^{(0)}$ is defined by the Christoffel symbols
\begin{equation}  \label{k1gamma}
 \Gamma^l_{(0)ij} = -g_{i i'} b^{i' l}_j .
\end{equation}

Our main theorem in this case just states that the connection $\nabla^{(0)}$ is flat. As it is well-known from the complete description given by Dubrovin and Novikov in~\cite{dn83}, the bracket~\eqref{pb-deg-1} is Poisson if and only if the connection $\nabla^{(0)}$ is flat, torsionless, and compatible with the metric $g$. Notice that, while the first two conditions are true for $k\geq1$, the compatibility of $\nabla^{(0)}$ with $g$ does not hold in general. However, it can be proved from the skew-symmetry constraint that the connection obtained from $\nabla^{(k-1)}$ by changing the sign of its torsion is compatible with $g$ for $k\geq1$. For $k=1$ this gives the remaining constraint.

\subsection{Poisson brackets of degree $k=2$}

The Poisson brackets of  degree $k=2$ in local coordinates have the form
\begin{equation} \label{pb-deg-2}
\{ u^i(x) ,  u^j(y) \} = g^{ij} \delta''(x-y) + b^{ij}_l u^l_x \delta'(x-y) + ( c^{ij}_l u^l_{xx} + c^{ij}_{lm} u^l_x u^m_x ) \delta(x-y)
\end{equation}
where we assume that the bivector $g$ is non-degenerate, therefore the dimension $n$ of $M$ is even.  
The Christoffel symbols of the standard connections are given by
\begin{equation}
 \Gamma^l_{(0)ij} = -g_{i i'} c^{i' l}_j,\qquad 
\Gamma^l_{(1)ij} = - \frac12 g_{i i'} b^{i' l}_j,
\end{equation}
and $\Gamma_{[1]} = 3 \Gamma_{(0)} - 2 \Gamma_{(1)}$, namely
\begin{equation}
\Gamma^l_{[1]ij} = g_{i i'} (b^{i' l}_j-3 c^{i' l}_j ).
\end{equation}
Our main theorem asserts that the connections $\nabla^{(0)}$ and $\nabla^{[1]}$ are flat. 

Ferguson gives in~\cite{f09} a complete set of equations for the skew-symmetry and Jacobi identity of the Poisson bracket. We recall this result in the following equivalent form.
\begin{theorem}[\cite{f09}] \label{fergu-thm}
For $g$ non-degenerate, formula~\eqref{pb-deg-2} defines a Poisson bracket if and only if 
\begin{enumerate}[(a)]
\item $g^{ij}$ is skew-symmetric,
\item $\nabla^{(0)}$ is flat and torsionless,
\item $\nabla^{(0)}_{i} g_{jl}$ is skew-symmetric in $i,j,l$,
\item $\nabla^{(0)}_l  g^{ij} = b^{ij}_l - 2 c^{ij}_l$,
\item $c^{ij}_{ql} = c^{ij}_{(q,l)} - g_{pr} c^{ri}_{(q} c^{pj}_{l)}$. \end{enumerate}
\end{theorem}
Notice that the tensor appearing on the right-hand side of (d) is given by the difference of the flat connections
\begin{equation}\label{eq:pfF-1}
b^{ij}_l - 2 c^{ij}_l = g^{ir}(\Gamma^j_{[1]rl}-\Gamma^j_{(0)rl}).
\end{equation}

The flatness of $\nabla^{[1]}$ follows from Ferguson's equations but was previously unnoticed. 
\begin{corollary}
It follows from (a)-(d)  that the connection $\nabla^{[1]}$ is flat. 
\end{corollary}
\begin{proof}
Notice that we can write $\nabla^{[1]} = \nabla^{(0)} + A$, where $A$ is the tensor obtained from the difference of $\Gamma_{[1]}$ and $\Gamma_{(0)}$ in \eqref{eq:pfF-1}, which from (d) of Theorem~\ref{fergu-thm} reads
\begin{equation}
A^l_{ij}=g_{ir}\nabla^{(0)}_jg^{rl}.
\end{equation}
Using (c), the same tensor can equivalently be written as
\begin{equation}
A^l_{ij} = g^{ls} \nabla^{(0)}_{i} g_{sj}.
\end{equation}
By the flatness of $\nabla^{(0)}$, we can compute the Riemann curvature of $\nabla^{[1]}$ as
\begin{align}
[\nabla^{[1]}_i,\nabla^{[1]}_j]^s_t &= \nabla^{(0)}_i A^{s}_{jt} + A^s_{iq} A^q_{jt} - ( i \leftrightarrow j ) \notag \\
&=\nabla_i^{(0)}\left(g^{sl}\nabla_j^{(0)}g_{lt}\right)+g^{sp}\nabla^{(0)}_ig_{pq}g^{ql}\nabla^{(0)}_jg_{lt}-( i \leftrightarrow j ) \notag\\
&= g^{sl} \nabla^{(0)}_i \nabla^{(0)}_j g_{lt} - ( i \leftrightarrow j ) \notag \\
&= g^{sl} [\nabla^{(0)}_i, \nabla^{(0)}_j] g_{lt} = 0. \notag
\end{align}
\end{proof}

\begin{remark}
The connections are in general position if $\nabla^{[1]}$ and $\nabla^{(0)}$ do not coincide, equivalently if $b^{ij}_l \not= 2 c^{ij}_l$. The degenerate case $\nabla^{[1]}=\nabla^{(0)}$, namely $b^{ij}_l = 2 c^{ij}_l$, has been also considered in~\cite{f08} and corresponds to the case of covariantly constant $g$, which in turn is equivalent to the fact that $g$ defines a Poisson tensor on $M$. 
\end{remark}

\begin{remark}
As it follows immediately from Ferguson's equations, in the flat coordinates for $\nabla^{(0)}$ the Poisson operator takes the canonical form 
\begin{equation}
P^{ij} = \partial g^{ij} \partial
\end{equation}
and, equivalently, the Poisson bracket is
\begin{equation} 
\{ u^i(x) ,  u^j(y) \} = g^{ij} \delta''(x-y) + \partial_l g^{ij} u^l_x \delta'(x-y) 
\end{equation}
where $\partial_i g_{jk}$ is skew-symmetric in $i, j, k$, which in turn implies that $g_{jk}$ is linear in the flat coordinates $u$. 
\end{remark}

\begin{remark} \label{nonflat1}
In flat coordinates the Christoffel symbols $\Gamma^j_{[1]rl}$ and $\Gamma^j_{(0)rl}$ differ by a factor $-1/2$, and since we know that $\nabla^{[1]}$ is flat, as long as $g$ is not constant, $\nabla^{(1)}$ cannot be flat. 
\end{remark}

\subsection{Poisson brackets of degree $k=3$}

The Poisson bracket in this case is usually written in local coordinates as
\begin{multline} \label{poisson3}
\{ u^i(x) ,  u^j(y) \} = g^{ij} \delta'''(x-y) + b^{ij}_l u^l_x \delta''(x-y) + ( c^{ij}_l u^l_{xx} + c^{ij}_{lm} u^l_x u^m_x ) \delta'(x-y) \\
+(d^{ij}_l u^l_{xxx} + d^{ij}_{lm} u^l_{xx} u^m_x + d^{ij}_{lmn} u^l_x u^m_x u^n_x ) \delta(x-y).
\end{multline}
The non-degenerate matrix $g^{ij}$ is symmetric and transforms as a contravariant tensor. 
The Christoffel symbols of the standard connections are given by
\begin{equation}
 \Gamma^l_{(0)ij} = -g_{i i'} d^{i' l}_j,\qquad 
\Gamma^l_{(1)ij} = - \frac13 g_{i i'}  c^{i' l}_j, \qquad 
\Gamma^l_{(2)ij} = - \frac13 g_{i i'} b^{i' l}_j. 
\end{equation}

Our main theorem implies that the connections defined by the Christoffel symbols 
\begin{equation}
\Gamma^l_{[1]ij} = g_{i i'} ( c^{i' l}_j -4 d^{i' l}_j), \qquad 
\Gamma^l_{[2]ij} = g_{i i'} (-b^{i' l}_j+4 c^{i' l}_j -10 d^{i' l}_j),
\end{equation}
and the connection $\nabla^{(0)}$, are flat. 

The full set of constraints coming from skew-symmetry and Jacobi identity written in an arbitrary coordinate system is not available in the literature. However, according to Potëmin~\cite{p97}, in flat coordinates for the connection $\nabla^{(0)}$, the Poisson operator has the form
\begin{equation} \label{Pk3}
P^{ij} = \partial \left( g^{ij} \partial + c^{ij}_l u^{l}_x \right) \partial
\end{equation}
and the full set of constraints reduces to the following equations
\begin{subequations}\label{potemineq}
\begin{align}  
g^{ij}_{,l} &= c^{ij}_l + c^{ji}_l,\\
g^{is}c^{jl}_s&=-g^{js}c^{il}_s,\\
0&=g^{is}c^{jl}_s+g^{js}c^{li}_s+g^{ls}c^{ij}_s,\\
g^{ls}c^{ij}_{s,m}&=c^{il}_sc^{sj}_m-c^{li}_sc^{sj}_m-c^{lj}_sg^{si}_{,m}.
\end{align}
\end{subequations}

We can read the coefficients in~\eqref{poisson3} by expanding~\eqref{Pk3} and comparing the two expressions. Besides $g^{ij}$ and $c^{ij}_l$ which represent the same objects in the two formulas, we have
\begin{equation}
b^{ij}_l = \frac{\partial g^{ij}}{\partial u^l} + c^{ij}_l, \qquad 
c^{ij}_{lm} = c^{ij}_{(l,m)}, 
\end{equation}
while all the coefficients $d^{ij}_{l}$, $d^{ij}_{lm}$ and $d^{ij}_{lmp}$ vanish.

The Christoffel symbols of the flat connections $\nabla^{[1]}$ and $\nabla^{[2]}$ are
\begin{equation}
\Gamma_{[1]ij}^l = g_{is} c^{sl}_j, \qquad 
\Gamma_{[2]ij}^l =   2 g_{is} c^{sl}_j- g_{is} c^{ls}_j.
\end{equation}
The flatness of the connection $\nabla^{[1]}$ was proved by Balandin and Potëmin~\cite{bp01}. The flatness of the connection $\nabla^{[2]}$ was so far unnoticed, but it can be directly verified in these coordinates by quite long but straightforward computations using equations~\eqref{potemineq}.

\begin{remark}
The connections are in general position if both $b^{ij}_l$ and $c^{ij}_l$ are non-zero and not proportional to each other, equivalently if $\nabla^{[1]}$ and $\nabla^{[2]}$ are not proportional and both non-zero. The non-generic cases are either the case of $\nabla^{[1]}=\nabla^{[2]}$ non-vanishing, which corresponds to $c^{ij}_l$ symmetric in the upper indices, or the completely degenerate case, which corresponds to vanishing connections and constant $g^{ij}$.
\end{remark}

\begin{remark} \label{nonflat2}
Let us consider the degree $k=3$ Poisson bracket in $n=2$ dimensions considered in Theorem 1 of~\cite{fpv14} which is given by
\begin{equation}
(g^{ij}) = \begin{pmatrix}
1 & \frac{u^2}{u^1}  \\  \frac{u^2}{u^1} & \frac{1 + (u^2)^2}{(u^1)^2} 
\end{pmatrix}, \quad 
(c_1^{ij}) = \begin{pmatrix}
0 & -\frac{u^2}{(u^1)^2} \\ 0 & - \frac{1 + (u^2)^2}{(u^1)^3}  
\end{pmatrix}, \quad 
(c_2^{ij}) = \begin{pmatrix}
0 & \frac{1}{u^1} \\ 0 & \frac{u^2}{(u^1)^2}  
\end{pmatrix}.  \notag
\end{equation}
Computing the curvature tensors of the connections $\nabla^{(1)}$ and $\nabla^{(2)}$ we obtain 
\begin{align}
R^{(1) 2}_{2,1,1}  &= - R^{(1) 2}_{1,2,1} = \frac{4}{9 (u^1)^2} , \\
R^{(2) 1}_{1,2,1}  &= - R^{(2) 1}_{2,1,1} = R^{(2) 2}_{2,1,2}  = - R^{(2) 2}_{1,2,2}= \frac{8 u^2}{9 u^1} \\
R^{(2) 2}_{1,2,1}  &= - R^{(2) 2}_{2,1,1} =  \frac{4 (2 (u^2)^2 -3)}{9 (u^1)^2}, \\
R^{(2) 1}_{2,1,2} &= - R^{(2) 1}_{1,2,2}  = \frac{8}{9}.
\end{align}
while all the other components are vanishing. This shows that in general the standard connections $\nabla^{(1)}$ and $\nabla^{(2)}$ have non-vanishing curvature. Notice that in these coordinates the Christoffel symbols $\Gamma^l_{[1]ij}$ are proportional to $\Gamma^l_{(1)ij}$: one can indeed check that, by rescaling the latter in the previous computation by the coefficient $-1/3$, the curvature vanishes. 
\end{remark}

\subsection{Poisson brackets of degree $k=4$}

The form of the Poisson bracket in local coordinates is 
\begin{multline}
\{ u^i(x) ,  u^j(y) \} = g^{ij} \delta''''(x-y) + b^{ij}_l u^l_x \delta'''(x-y) + ( c^{ij}_l u^l_{xx} + c^{ij}_{lm} u^l_x u^m_x ) \delta''(x-y) \\
+(d^{ij}_l u^l_{xxx} + d^{ij}_{lm} u^l_{xx} u^m_x + d^{ij}_{lmp} u^l_x u^m_x u^p_x ) \delta'(x-y) \\
+(e^{ij}_{l} u^l_{xxxx} + e^{ij}_{lm} u^l_{xxx} u^m_x + \hat{e}^{ij}_{lm} u^l_{xx} u^m_{xx} + e^{ij}_{lmp} u^l_{xx} u^m_x u^p_x + e^{ij}_{lmpq} u^l_{x} u^m_x u^p_x u^q_x
) \delta(x-y) .
\end{multline}
The matrix $g^{ij}$ is skew-symmetric and defines a bivector on the manifold $M$. We assume that $g^{ij}$ is non-degenerate, and therefore that the dimension $n$ of $M$ is even. 

The Christoffel symbols of the standard connections are given by
\begin{align}
 \Gamma^l_{(0)ij} &= -g_{i i'} e^{i' l}_j,\quad 
&\Gamma^l_{(1)ij} &= - \frac14 g_{i i'}  d^{i' l}_j, \\ 
\Gamma^l_{(2)ij} &= - \frac16 g_{i i'} c^{i' l}_j, \quad 
&\Gamma^l_{(3)ij} &= - \frac14 g_{i i'} b^{i' l}_j. 
\end{align}
Our main theorem implies that the connections defined by the Christoffel symbols 
\begin{align}
\Gamma^l_{[1]ij} &= g_{i i'} ( d^{i' l}_j -5 e^{i' l}_j), \\ 
\Gamma^l_{[2]ij} &= g_{i i'} (-c^{i' l}_j+5 d^{i' l}_j -15 e^{i' l}_j), \\ 
\Gamma^l_{[3]ij} &= g_{i i'} (b^{i' l}_j-5 c^{i' l}_j +15 d^{i' l}_j -35 e^{i' l}_j),
\end{align}
are flat. 

We are not aware of any result in the literature about the homogeneous Poisson brackets of degree $k=4$.

\section{Conclusions}

The aim of this work is to address the general problem of the study of the structure and deformations of homogeneous local Poisson brackets of arbitrary degree as geometric objects on a manifold $M$, as originally posed by Dubrovin and Novikov. 
By leveraging homological algebra techniques, we set the stage for a systematic investigation of the structure and cohomology of such brackets, with the aim of generalising well-known results in low degree (for the hydrodynamic case $k=1$, where these problems are completely solved, at least in the case of invertible $g$, see~\cite{dn83, dz01, g02, dms02}).

Our main result is the proof that specific linear combinations of the $k$ standard connections associated with a homogeneous Poisson bracket of degree $k$ have vanishing curvature. We observe moreover how this was mostly unnoticed in the previous studies of low degree brackets. 

This result appears as a side product of the first steps in the computation of the Poisson cohomology of the homogeneous Poisson brackets. Indeed, we plan to extend this analysis in a subsequent work, where we compute the Poisson cohomology of a homogeneous Poisson bracket $P$ in terms of the Chevalley-Eilenberg cohomology of a Lie algebra naturally associated with $P$. However, a more direct proof of the flatness of the connections $\nabla^{[s]}$ remains an intriguing challenge. Such a proof could possibly shed light on the geometric origins of the coefficients defining the flat connections.

We deem that achieving a complete geometric characterization of the constraints imposed by skew-symmetry and the Jacobi identity could significantly advance the classification of homogeneous Poisson brackets. There are some further directions of study of the structure problem that could be pursued.
The skew-symmetry constraint of the Poisson bracket has been used in this work (see proof of Lemma~\ref{lemma10}) but its full interpretation is lacking. Other properties, like the symmetry of $\nabla^{(0)}$, are known~\cite{d93} to hold for arbitrary degree but their proof escapes the methods used in this paper. The nature of ``deeper'' coefficients, those appearing at higher degree $\deg_u$ in the brackets, is still not clear. 
Finally, relaxing the assumption of invertibility of $g$, especially for even degrees $k$, could expand the scope of the theory to include degenerate cases, potentially uncovering new structural phenomena.

\end{document}